\documentclass[12pt,letterpaper]{amsart}

\usepackage{amssymb,latexsym,amsmath,amsthm,graphicx}
\numberwithin{equation}{section}
\newtheorem{thm}{Theorem}[section]
\newtheorem{lemma}[thm]{Lemma}

\newtheorem{prop}[thm]{Proposition}

\newtheorem*{remark}{Remark}

\usepackage{xcolor}
\newcommand\todo[1]{\textcolor{red}{#1}}
\newcommand\comm[1]{\textcolor{blue}{#1}}

\newcommand{\pd}{\partial}

\newcommand{\ol}{\overline}
\newcommand{\e}{\varepsilon}

\def \CC{\mathbb{C}}
\def \RR{\mathbb{R}}

\def \f{\tilde{f}}
\def \g{\tilde{g}}
\def \h{\tilde{h}}
\def \z{\tilde{\zeta}}
\def \F{\tilde{F}}
\def \p{\partial}

\begin{document}

\title[minimal graphs]{On the growth of solutions to the minimal surface equation over domains containing a halfplane}

\author[Erik Lundberg]{Erik Lundberg}

\author[Allen Weitsman]{Allen Weitsman}

\address{Department of Mathematics\\ Purdue University\\
West Lafayette, IN 47907-1395}
\address{Email: elundber@math.purdue.edu}
\address{Email: weitsman@purdue.edu}

%\subjclass[2000]{Primary: , Secondary: }
%\date{2011}

\begin{abstract}
We consider minimal graphs $u = u(x,y) > 0$ over unbounded domains $D$ with
$u = 0$ on $\pd D$.  
Assuming $D$ contains a sector properly containing a halfplane, we obtain estimates on growth
and provide examples illustrating a range of growth.


\end{abstract}

\maketitle

 \medskip
{\small
\noindent Corresponding author: Erik Lundberg \\
email: elundber@math.purdue.edu, elundberg9@gmail.com \\
phone: (765) 494-1965 \\
fax: (765) 494-0548}

%\tableofcontents

\section{Introduction} Let $D$ be an unbounded plane domain. 
  In this paper we consider the boundary value problem
for the minimal surface equation
\begin{equation}
\label{eq:bdryvalueprob}
\left\{
\begin{aligned}
  &\text{div} \frac{\nabla u}{\sqrt{1+|\nabla u|^2}}=0 \quad \text{and } u>0 \quad \text{in}\ D\\
        & u=0\quad \text{on}\ \partial D
 \end{aligned}\right.
 \end{equation}
 We shall study the constraints on growth of nontrivial solutions to (\ref{eq:bdryvalueprob}) as determined
 by the maximum 
 $$
 M(r)={\max}\ u(x,y),
 $$ 
where the max is taken over the values $r=\sqrt{x^2+y^2}$ and $(x,y)\in D$.

 Perhaps the first relevant theorem in this direction was proved by Nitsche  \cite[p. 256]{Nitsche} who 
observed that if $D$ is contained in a sector of opening strictly less than $\pi$, 
 then $u\equiv 0$.  For domains contained in a half plane, but not contained in any such sector, there
 are a host of solutions to (\ref{eq:bdryvalueprob}) which will be discussed later.  However, in this case,
 it has been shown \cite{Weitsman2005} that if $D$ is bounded by a Jordan arc,
$$
  Cr\leq M(r)\leq e^{Cr}\quad (r>r_0)
 $$
 for some positive constants $C$ and $r_0$.  
 
 If, on the other hand, the domain $D$ contains a sector of opening $\alpha$ bigger than $\pi$, we shall show that
 the growth of $M(r)$ is at most linear (see Theorem \ref{thm:bound} in Section 2).  
 Regarding the bound from below, with the order 
 $\rho$ of $u$ defined
 by
 $$
 \rho=\underset {r\to\infty}\lim\sup\frac{\log M(r)}{\log r},
 $$
 it follows 
 by using the module estimates of Miklyukov \cite{Miklyukov} (see also chapter 9 in \cite{Miklyukovbook}) 
 as in \cite{Weitsman2005a} that if $D$ omits a sector of opening $2\pi -\alpha $, ($\pi\leq\alpha\leq 2\pi$,
the omitted set in the case $\alpha=2\pi$ being a line), then
the order of any nontrivial solution to (\ref{eq:bdryvalueprob}) is at least
$\pi/\alpha$,
 
 The paper concludes with a list of problems and conjectures.

\section{Estimates on Growth}

For later convenience we shall use complex notation $z=x+iy$ for points $(x,y)$ when describing solutions to
the minimal surface equation.  As such, 
we are given a minimal graph with positive height function $u(z)$ over a domain $D$ as in (\ref{eq:bdryvalueprob}).

\begin{thm}\label{thm:bound}
 Let $D$ be a domain  whose boundary is a Jordan arc, and $D$ contains a sector $S_\lambda := \{ z: |\arg z| \leq \lambda\}$, with $\lambda > \pi / 2$. 
 With $M(r)$ defined as above, if $u$ satisfies (\ref{eq:bdryvalueprob}) in $D$ and vanishes on the boundary $\pd D$, then there exist positive constants $K$ and $R$ such that
\begin{equation}
\label{eq:bounds}
	 M(r) \leq K r, \quad |z| > R.
\end{equation}
\end{thm}

Throughout, we will make use of the parametrization in isothermal coordinates by the 
Weierstrass functions $\left( x(\zeta), y (\zeta ), U (\zeta) \right)$
with $\zeta$ in the right half plane $H$, $U(\zeta) = u(x(\zeta),y(\zeta))$ and (up to additive constants)

\begin{equation}
\label{eq:Weierstrass}
\left\{
\begin{aligned}
	x(\zeta) &= \Re e \, \frac{1}{2} \int_{\zeta_0}^{\zeta} \omega(\bar{\zeta}) (1-G^2(\bar{\zeta})) d \bar{\zeta} \\
	y(\zeta) &= \Re e \, \frac{i}{2} \int_{\zeta_0}^{\zeta} \omega(\bar{\zeta}) (1+G^2(\bar{\zeta})) d \bar{\zeta} \\
	U(\zeta) &= \Re e \, \int_{\zeta_0}^{\zeta} \omega(\bar{\zeta}) G(\bar{\zeta}) d \bar{\zeta} \\
\end{aligned}\right.
\end{equation}

With this parameterization, the height function $U(\zeta)$ pulled back to the halfplane $H$
becomes a positive harmonic function in $H$ which is 0 on the imaginary 
axis, and thus  is simply $U(\zeta) = C \Re e \{ \zeta \}$ for a real positive constant $C$.
We may assume without loss of generality that $C=2$.

Since $f(\zeta ) := x(\zeta )+iy(\zeta )$ is harmonic in $H$, there exist analytic functions
$h(\zeta )$ and $g(\zeta )$ in $H$ such that 

$$f(\zeta) = h(\zeta) + \ol{g(\zeta)}.$$
With this formulation, the height function then satisfies
$$
U(\zeta )= 2\Re e\, i\int \sqrt{h'(\zeta )g'(\zeta )}\,d\zeta ,
$$
and since $U(\zeta )=\zeta $ in (\ref{eq:Weierstrass}), it follows  that
\begin{equation}\label{eq:diff}
 g'(\zeta) = - \frac{1}{h'(\zeta)}.
\end{equation}

\subsection{Proof of Theorem \ref{thm:bound}}

First we establish the bound (\ref{eq:bounds}) inside a sector.

\begin{lemma}\label{lemma:angle}
 Let $S_\alpha := \{ z: |\arg z | \leq \alpha < \pi/2 \}$ be a sector contained in $H \subset D$.
 Then for some $K>0$ the upper bound (\ref{eq:bounds}) holds in $S_\alpha$ for all $r$ sufficiently large:
 $$\max_{|z|=r, z \in S_\alpha} u(z) \leq Kr.$$
\end{lemma}

\begin{proof}[Proof of Lemma]

Let $f(\zeta)$, $U(\zeta)$ be as above.
So, $u(f(\zeta)) = U(\zeta) = \Re e \, \zeta $.

Let $P:= \{ \zeta: \Re e \, f(\zeta) > 0 \}$ be the preimage of the right halfplane,
and introduce a new variable $\z$ and let $\psi(\z)$ be a conformal map 
from the right half $\z$-plane $H := \{\z: \Re e \, (\z)>0 \}$ onto $P$.

Define
\begin{equation*}
\left\{
\begin{array}{l}
\f(\z) := f(\psi(\z)) \\
\g(\z) := g(\psi(\z)) \\
\h(\z) := h(\psi(\z))
\end{array}\right. 
\end{equation*}

Then $\f$ is a harmonic map, and 
$$ \f(\z) = \h(\z) + \ol{\g(\z)}.$$

We wish to show that for all $|z|>R$ in $S_\alpha$,
$$\frac{u(z)}{|z|} = \frac{U(\zeta)}{|f(\zeta)|} = \frac{\Re e \, \zeta}{|f(\zeta)|} = \frac{\Re e \, \psi(\z)}{|\f(\z) |} < K.$$

Let $\F(\z) = \h(\z) + \g(\z)$ be the analytic function with the same real part as $\f$. 
Then $\Re e \, \F$ is positive in $H$ and vanishes on $\p H$, and therefore, without loss of generality we may write (see \cite[p. 151]{Tsuji})
\begin{equation}\label{eq:F}
 \F(\z) = \z \implies \F'(\z) = 1.
\end{equation}

The proof hinges on (\ref{eq:F}) along with the chain rule combined with (\ref{eq:diff}).  Now, 
$$ \h'(\z) = h'(\psi(\z)) \cdot \psi'(\z). $$
and
\begin{equation}
\label{eq:g'}
 \g'(\z) = - \frac{\psi'(\z)}{h'(\psi(\z)) } = - \frac{\psi'(\z)^2}{\h'(\z)}. 
\end{equation}

Combining this with (\ref{eq:F}) we have
\begin{equation*}
 1 = \F'(\z) = \h'(\z) - \frac{\psi'(\z)^2}{\h'(\z)}  
\end{equation*}
which implies
\begin{equation*}
 \h'(\z)^2 - \h'(\z) - \psi'(\z)^2 = 0.
\end{equation*}
Thus,
\begin{equation}  \label{eq:h'} 
\h'(\z) = \frac{1 + \sqrt{1+4\psi'(\z)^2}}{2}.
\end{equation}

Since $\psi(\z)$ is a conformal map with $\Re e \, \psi(\z) > 0$ in $H$, there exists a real constant
$0 \leq c < \infty$ such that in any sector $S_\beta := \{ \z: |\arg \z | \leq \beta < \pi/2 \}$ the limit $\psi'(\z) \rightarrow c$ exists as $\z \rightarrow \infty$ in $S_\beta$. (see \cite[p. 152]{Tsuji})

{\bf Case 1:} $\psi'(\z) \rightarrow c= 0$ as $\z \rightarrow \infty$ (with $\z$ in $S_\beta$).

From (\ref{eq:h'}) we have $\h'(\z) \rightarrow 1$ as $\z \rightarrow \infty$, and using (\ref{eq:g'}) we have $\g'(\z) \rightarrow 0.$
Thus, $\h(\z) \approx \z$ and $\g(\z) = o(1)$,
which implies that $\f(\z) = \h(\z) + \ol{\g(\z)} \approx \z$.

Since $\f : H \rightarrow H$ is asymptotic to the identity map,
given $\alpha$, we may choose $\beta < \pi/2$ so that $S_\alpha \cap  \{ |z| > R \}$ is contained in the image of the sector $S_\beta$
for $R$ large enough.
Thus, the estimate $\psi'(\z) \rightarrow  0$ applies in the region $S_\alpha$, and we have
$$\frac{u(z)}{|z|}=\frac{\Re e \, \psi(\z)}{|\f(\z) |} < \frac{|\psi(\z)|}{|\f(\z) |} = o(1), \quad \text{for } z \in S_\alpha \cap  \{ |z| > R \},$$
since $\f(\z) \approx \z$, and $\psi'(\z) = o(1).$

{\bf Case 2:} $\psi'(\z) \rightarrow c > 0$ as $\z \rightarrow \infty$.

From (\ref{eq:F}) we have $\Re e \, \{\h(\z) + \ol{\g(\z)} \} = \Re e \, \z$.
Let us also estimate $\Im m \, \f(\z) = \Im m \, \h(\z) - \Im m\, \g(\z)$.
We use (\ref{eq:h'}) and (\ref{eq:g'}):
$$ \h'(\z) \rightarrow \frac{1+\sqrt{1+4c^2}}{2}, $$
$$ \g'(\z) \rightarrow \frac{-2 c^2}{1+\sqrt{1+4c^2}}, $$
which imply
$$\h'(\z) - \g'(\z) \rightarrow \frac{(1+\sqrt{1+4c})^2 + 4c^2}{2(1+\sqrt{1+4c^2})} = 1 + \frac{4c^2}{1+\sqrt{1+4c^2}} .$$
Putting this together, we have
$$ \h(\z) + \ol{\g(\z)} = \Re e \, \z + i \left( 1 + \frac{4c^2}{1+\sqrt{1+4c^2}} + o(1) \right) \Im m\, \z.$$

As in the first case, given $\alpha$, we may thus choose $\beta < \pi/2$ and $R>0$ so that $S_\alpha \cap  \{ |z| > R \}$ is contained in the image $\f(S_\beta)$ of the sector $S_\beta$.
Then we have
$$\frac{u(z)}{|z|} = \frac{\Re e \, \psi(\z)}{|\f(\z) |} < \frac{|\psi(\z)|}{|\f(\z) |} = O(1), \quad \text{for } z \in S_\alpha \cap  \{ |z| > R \}.$$
Indeed, $\displaystyle |\f(\z)| =\left|\Re e \, \z + i\left(1 + \frac{4c^2}{1+\sqrt{1+4c^2}} + o(1)\right) \Im m\, \z \right|$, and 
$\psi'(\z) = O(1) \implies \psi(\z) = O(|\z|).$

\end{proof}

 Applying Lemma \ref{lemma:angle} to two sectors, one rotated clockwise and the other counterclockwise, in order that their union covers $S_\lambda$,
 the upper bound (\ref{eq:bounds}) is established in $S_\lambda$.  It remains to prove the estimate in the rest of $D$.

Let $\pi/2 < \alpha < \lambda$.  We will show that the upper bound (\ref{eq:bounds}) holds in $D \setminus \ol{S_\alpha}$.

In order to prove this, we will apply the following result from \cite[Main Theorem]{Hwang88}:

{\bf Theorem A.}  
Let $\Omega \subset \Omega_1 = \{ (x,y): x>0, -f(x) < y < f(x) \},$ 
where $f,g \in C[0,\infty)$, $f,g \geq 0$, $g(0)=0$, $f(t), g(t)/t$ increase as $t$ increases, and let $u \in C(\ol{\Omega}) \cap C^2(\Omega)$.
Suppose that

i) {div}$\displaystyle{\frac{\nabla u}{\sqrt{1+|\nabla u|^2}}} \geq 0$ in $\Omega$,

ii) $u|_{ \partial \Omega \cap ( \{x\} \times [-f(x),f(x)] )} \leq g(x)$ for $x \in [0,\infty)$,

iii) $0 < \kappa(x) := f(x) / (2 g(x)) < 1$ for all $x$ larger than some $x_1 > 0$,

iv) $\kappa(x)$ is decreasing on $[x_1,\infty)$.

Then $u(x,y) \leq g( x / (1 - \kappa(x)))$ for every $(x,y) \in \Omega$ with $x > x_1$.
\vskip .3truein
We apply this to $\Omega = D \setminus \ol{S_\alpha}$, while taking $\Omega_1 = \CC \setminus \ol{S_\alpha}$.  
In order to relate to the setup in the theorem, reflect these domains about the $y$-axis, 
so that $\Omega$ and $\Omega_1$ are in the right halfplane.
Then $\Omega_1 = \{ (x,y) : x > 0, -f(x) < y < f(x) \}$, where $f(x) = \tan(\pi - \alpha) x$.  If $C>0$ is sufficiently large, then $g(x) = C x(1 - \exp(-x)/2)$ satisfies both (iii) and (iv).
We check that for $C$ large enough, (ii) is also satisfied.  Note that $\p \Omega$ contains points on $\p D$ and points on $\p S_\alpha$.
For points on $\p D$, $u=0$, and for points on $\p S_\alpha$, $u$ has at most linear growth by Lemma \ref{lemma:angle}.
Thus, in both cases (ii) is satisfied, and Theorem A may be applied.
The result is that $u(x,y) \leq g(x / (1 - \kappa(x)))$ for all large enough $x \in \Omega$.
Since 
$$\frac{x}{1 - \kappa(x)} = \frac{x}{1-\tan(\pi-\alpha)/C} (1+o(1)) ,$$
and $\tan(\pi-\alpha)/C$ is a small constant provided $C$ is large,
we have
$$u(x,y) < Cx,$$
for all large enough $x \in \Omega$.
This completes the proof of  (\ref{eq:bounds}).

\subsection{A lower bound}

\begin{prop}\label{prop:lower}
Suppose $D$ is a domain with $\pd D \neq \emptyset$,
and $u(z)>0$ satisfies (\ref{eq:bdryvalueprob}) with $u(z)=0$ on $\pd D$.
Then $u(z)$ has at least logarithmic growth.
\end{prop}

\begin{proof}
Without loss of generality assume that $0 \in \pd D$,
and consider the top half of the vertical catenoid centered at $z=0$ as a ``barrier'' (cf. \cite[p. 92]{Osserman})
Explicitly, let $\cosh^{-1}$ denote the positive branch of the inverse of $\cosh :\RR \rightarrow \RR$,
and define
$$G(z;r_1) := r_1 \cosh^{-1} \left( \frac{|z|}{r_1} \right), \quad |z| \geq r_1 .$$
For each $r_1$, $G(z;r_1)$ satisfies (\ref{eq:bdryvalueprob}).

Let  $\e>0$ and choose a $\delta$-neighborhood $B(\delta,0)$ of $z=0$ small enough that $u(z) < \e$ throughout $B(\delta,0) \cap D$.

Define $u_\e(z) = u(z) - \e$.  For $r_1>0$ small enough, $G(|z|;r_1)> u_\e(z)$ on $\pd B(\delta,0) \cap D$.
For $R>0$, let $$K_R :=  D \cap B(R,0)  \setminus B(\delta,0).$$ 

Fix $R=R_0$.  Suppose $\max_{|z|=R} |u(z)|$ grows slower than logarithmically, so it grows slower than $G(|z|,r_1)$.  
Then for $R>R_0$ sufficiently large, $G(|z|;r_1) > u_\e(z)$ on $\pd K_R$.
This implies the same inequality throughout $K_{R_0} \subset K_R$.
In particular, $u_\e(z) < r_1 \cosh^{-1} \left( \frac{R_0}{r_1} \right)$ in $K_{R_0}$.
But $r_1>0$ is arbitrary,
and $r_1 \cosh^{-1} \left( \frac{R_0}{r_1} \right) \rightarrow 0$ as $r_1 \rightarrow 0$.
Thus, $u_\e(z) \leq 0$ in $K_{R_0}$ which implies that $u(z) \leq 0$ since $\e$ was arbitrary.
This contradicts that $u(z)>0$ in $D$.
\end{proof}

\section{Examples}

In this Section, we provide examples that together with the above (and previously known) results give a broad picture
of the possible growth rates of minimal graphs.
One notices three ``regimes'' illustrated in Fig. \ref{fig:growth}.
When $D$ contains a halfplane we find nontrivial examples, but their growth rates appear to be determined by the asymptotic angle $\pi < \beta < 2 \pi$.
This is reminiscent of the behavior of positive harmonic functions,
hence we deem this the ``Phragm\'en-Lindel\"of regime''.  However, the geometry of $D$ plays
a subtle role, since if $D$ is a true sector of opening $\beta$, even in the range $\pi <\beta <2\pi$,
then (\ref{eq:bdryvalueprob}) has only the trivial solution $u\equiv 0$ \cite[p.993]{LLR}.

When $D$ is contained in a sector $\beta < \pi$, we have a ``completely rigid regime'', 
due to Nitsche's theorem.
At the critical angle $\beta = \pi$, an interesting phase transition occurs; 
there are examples with $D$ contained in a halfplane with $\beta = \pi$ exhibiting a full spectrum of possible growth rates anywhere from linear to exponential thus interpolating the known upper and lower bounds.

\begin{figure}[h]
    \begin{center}
    \includegraphics[scale=.22]{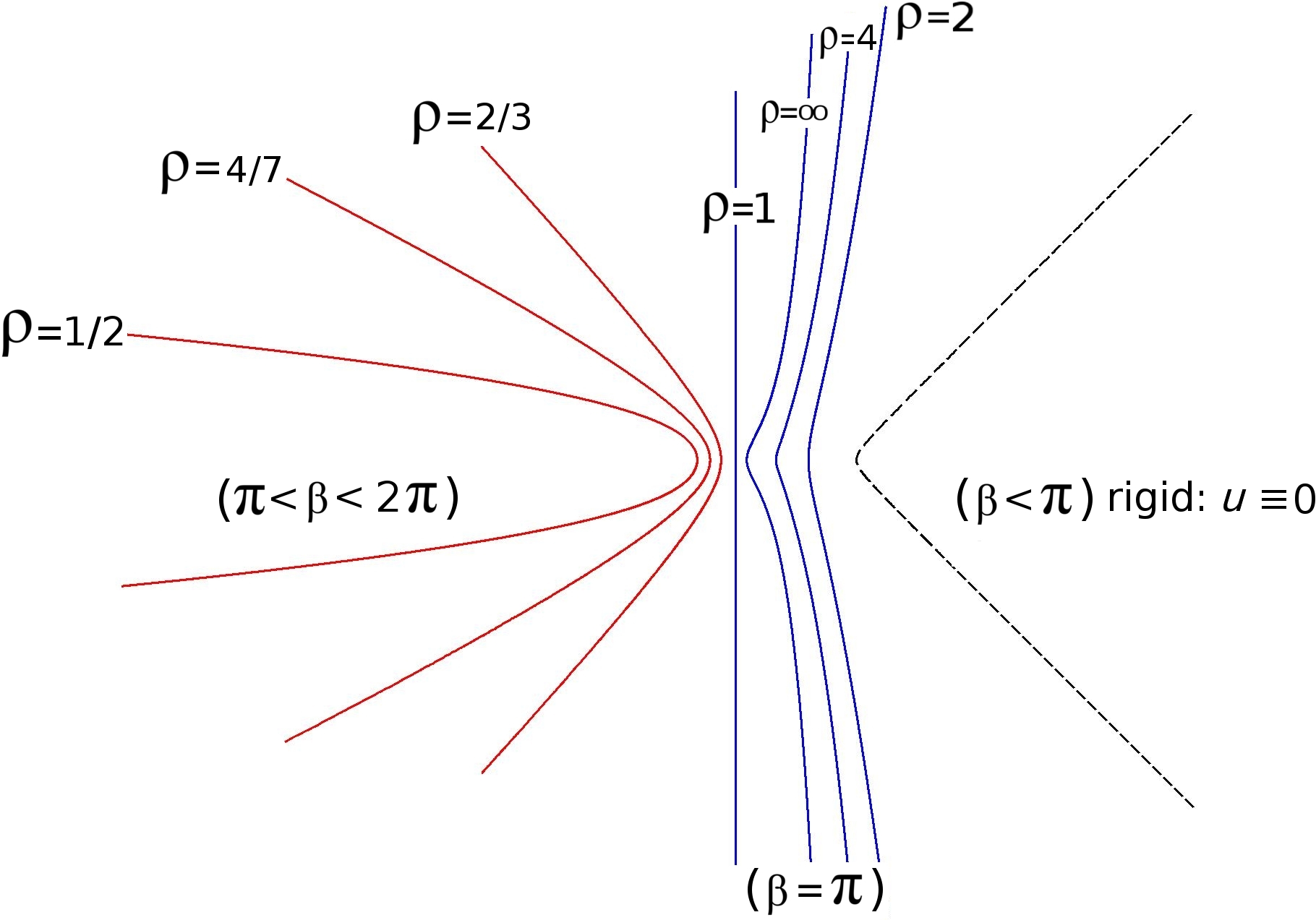}
    \end{center}
    \caption{A plot of the boundary of $D$ labeled with order $\rho$. \todo{Phragm\'en-Lindel\"of regime: $\pi < \beta < 2 \pi$}, \comm{Critical regime: $\beta=\pi$,} and Rigid regime: $\beta < \pi$.
    For the  curves, from left to right the angles are $\beta = 2 \pi$, $7\pi / 4$, and $3 \pi/2$. }
    \label{fig:growth}
\end{figure}

\subsection{Examples in the ``Phragm\'en-Lindel\"of'' regime $\pi < \beta < 2\pi$:}

In \cite{Weitsman2005}, there appears an example of a minimal graph with height function (pulled back to $\zeta$-plane) $U(\zeta) = 2 \Re e\, \zeta$,
and harmonic map from the half plane $H := \{z=x+iy:x>0 \}$
$$ z(\zeta) = \frac{(\zeta + 1)^2}{2} - \log(\bar{\zeta}+1) . $$
This example has asymptotic angle $2 \pi$ and growth of order $1/2$.  (See \S4 for the definition of asymptotic angle.)

Let us demonstrate a whole one-parameter family of examples with asymptotic angles $\pi < \beta < 2 \pi$ having growth of orders $\pi/\beta$.
Let $\gamma = \beta / \pi$ (so $1 < \gamma < 2$).
Then such a minimal surface is given by the harmonic map from the half plane $H$ to a region $D$
$$ z(\zeta) = (\zeta + 1)^{\gamma} - \frac{1}{\gamma(2 - \gamma)}(\bar{\zeta} + 1)^{2 - \gamma}$$
together with the height function $U(\zeta) = 2 \Re e\,  \zeta$.

Assuming $z(\zeta)$ is univalent, then we have growth of order $1 / \gamma = \pi / \beta$ as desired, since
$$\frac{u(z)}{|z|^{1/\gamma}} = \frac{U(\zeta)}{|z(\zeta)|^{1/\gamma}} = \frac{2\Re e \, \zeta}{| (\zeta + 1)^{\gamma} - \frac{1}{\gamma(2 - \gamma)}(\bar{\zeta} + 1)^{2 - \gamma} |^{1/\gamma}} .$$

Thus, the only thing to check is that $z(\zeta)$ is univalent in $H$.
Its Jacobian is 
$$\gamma^2 |\zeta + 1|^{2(\gamma-1)} - \frac{1}{\gamma^2 |\zeta + 1|^{2(\gamma - 1)}} > 0$$
since
$$\gamma^2 |\zeta + 1|^{2(\gamma-1)} > 1 .$$
Thus, global univalence can be ensured by checking the boundary behavior.
We will show that the imaginary part of $z(\zeta)$ is increasing on the boundary $\zeta = it$, $-\infty < t < \infty$.
The imaginary part of $z(it)$ is
$$ \Im m \, \{ z(it) \} = (1+t^2)^{\gamma / 2} \sin(\gamma \tan^{-1} t) + \frac{1}{\gamma(2 - \gamma)} (1+t^2)^{(2-\gamma)/2} \sin( (2-\gamma) \tan^{-1} t).$$
This is an odd function, so we just consider the interval $0 < t < \infty$.
The second term is increasing, since it is a product of increasing functions.
Indeed, $0 < 2 - \gamma < 1$ so that $0 < (2-\gamma) \tan^{-1} t < \pi / 2$ for $0 < t < \infty$
so that $\sin( (2-\gamma) \tan^{-1} t)$ is increasing.
In order to show that $(1+t^2)^{\gamma / 2} \sin(\gamma \tan^{-1} t)$ is increasing, we check that the derivative
$$ \gamma(1+t^2)^{\gamma / 2 - 1} t \sin(\gamma \tan^{-1} t) + \gamma(1+t^2)^{\gamma / 2 - 1} \cos(\gamma \tan^{-1} t)$$
is positive, or equivalently that
$$t \sin(\gamma \tan^{-1} t) + \cos(\gamma \tan^{-1} t) > 0 .$$
For this let $0 < \theta < \pi / 2$ and take $t = \tan \theta$.
Then we see that
$$\tan \theta \sin(\gamma \theta) + \cos(\gamma \theta) = \frac{\cos (\gamma-1) \theta}{\cos \theta}, $$
which is positive since $0 < \theta < \pi / 2$ and $1 < \gamma < 2$.

\subsection{The critical angle $\beta = \pi$:  Examples from linear growth to exponential.}

A plane and a horizontal catenoid sliced by a plane parallel to its axis provide two examples of minimal graphs over a domain contained in a half plane.  
These examples have linear and exponential growth respectively.

For each given $\rho>1$,
we provide an example contained in a halfplane (each having asymptotic angle $\beta=\pi$) with order of growth $\rho$.
Let $b=1/\rho$.  Then, once again, $z(\zeta)$ has the form 
$$ z(\zeta) = h(\zeta) - \ol{\int \frac{1}{h'(\zeta)}d\zeta},$$
so that $U(\zeta) = 2 \Re e \, \zeta$.

Taking $h(\zeta) = \zeta + \frac{1}{b}\zeta^b$,
$$ z(\zeta) = \zeta + \frac{1}{b}\zeta^b - \bar{\zeta} + \ol{\int \frac{1}{1+\zeta^{1-b}}d\zeta},$$

Assuming $z(\zeta)$ is univalent, $u(z)$ has order $\rho$, since
$$\frac{u(z)}{|z|^{\rho}} = \frac{U(\zeta)}{|z(\zeta)|^{\rho}} = \frac{2\Re e \, \zeta}{|z(\zeta) |^{\rho}},$$
which tends to a constant on the real axis.

It remains to check that $z(\zeta)$ is univalent in $H$.
Its Jacobian is 
$$|1 + \zeta^{b-1}|^2 - \frac{1}{|1 + \zeta^{b-1}|^2} > 0$$
since
$$|1 + \zeta^{b-1}|^2 > 1, \quad \text{for } \zeta \in H. $$
Thus, global univalence can be ensured by checking the boundary behavior.
As in the previous examples we show that $\Im m \,  \{ z(\zeta) \}$ is increasing on the boundary $\zeta = it$, $-\infty < t < \infty$.
This is an odd function, so we just consider the interval $0 < t < \infty$.
It suffices to show that the derivative 
$$\frac{d}{dt} \Im m \, \{ z(it) \}$$ 
is positive.
We use the identity
$$ \frac{d}{dt} \Im m \,  \{ z(it) \} = \frac{d}{dt} \Im m \,  \{ h(it)\} - \frac{d}{dt}\Im m \,  \{ g(it) \} =  \Re e \, \{ h'(it) \} - \Re e \, \{ g'(it) \} ,$$
to compute
\begin{equation*}
\begin{aligned}
\frac{d}{dt} \Im m \,  \{ z(it) \} &= 1 + \Re e \, \frac{1}{(it)^{1-b}} + 1 - \Re e \, \frac{1}{1+(it)^{1-b}} \\
&> 2 -  \frac{1}{1+\Re e \, \{ (it)^{1-b} \}} \\
&> 1.
 \end{aligned}
 \end{equation*}
We note that the domain $D$ for this example has a corner at the point $z(0)$.
This can be removed by shifting the minimal graph $(x,y,u(x,y))$ in the negative $u$-direction.

\section{Problems and conjectures}

\quad I.  When dealing with a nonlinear equation, issues of existence and uniqueness are often complex.  A survey
of uniqueness results can be found in \cite{Hwang03}.  A natural question to ask here is

{\bf Problem 1.}  Is it possible for (\ref{eq:bdryvalueprob}) to have more than one nontrivial solution?
\vskip .1truein
\quad II.  As discussed in the introduction, for domains $D$ contained in the half plane, at least when bounded by a Jordan
arc, the growth of solutions to (\ref{eq:bdryvalueprob}) is at most exponential.  However, it seems likely that this is true in general.

{\bf Problem 2.}  If $u$ is a solution to (\ref{eq:bdryvalueprob}), then does its maximum $M(r)$ satisfy
$$
M(r)\leq e^{Cr}  \quad (r>r_0),
$$
for some positive constants $C$ and $r_0$

\quad III.  In the case where $D$ contains a half plane, we have been unable to ascertain a good upper bound for the maximum.  However, it seems reasonable to conjecture that Theorem 2.1 holds here as well.

{\bf  Problem 3.}  If $u$ is a solution to (\ref{eq:bdryvalueprob}) and $D$ contains a half plane, then is it true that
$$
M(r)\leq Cr \quad (r>r_0)
$$
for some positive constants $C$ and $r_0$?
\vskip .1truein
\quad IV.  In this paper we have shown that if $D$ contains a sector of opening $\alpha >\pi$, then any nontrivial
solution has order at most 1.  However, it seems likely that this might be be improved.

{\bf Problem 4.}  If $D$ contains a sector of opening $\alpha >\pi$, then is it true that the order of any
nontrivial solution to
(\ref{eq:bdryvalueprob}) is bounded above by  $\pi /(2\pi - \alpha )$?  The interpretation as with the 
minimum bound discussed in \S1 has the case $\alpha - 2\pi$ taken to mean that the omitted set is a 
line.
\vskip .1truein
\quad V.  The results in \cite{Weitsman2005} are phrased in terms of the \emph{asymptotic angle} $\beta$
defined as follows.  Let $\Theta (r)$ be the angular measure of the set $D\cap\{|z|=r\}$, and 
$\Theta^*(r)=\Theta (r)$ if $D$ does not contain the circle $|z|=r$, and $+\infty$ otherwise.  Then 
$$
\beta =\underset{r\to\infty}{\lim\ \sup} \ \Theta^*(r).
$$
Consideration of the case $\beta=2\pi$ raises the following question

{\bf Problem 5.}  If $D$ is an unbounded region bounded by a Jordan arc (taken to mean a proper curve
which does not self intersect or close), then is it true that the  maximum of a nontrivial solution  satisfies
$$
M(r)\geq C\sqrt{r}\quad( r>r_0)
$$
for some positive constants $C$ and $r_0$?
\vskip .1truein
\quad VI.  Returning to Nitsche's theorem as mentioned in \S1, in terms of the asymptotic angle $\beta$ it seems likely
that a corresponding result should hold.

{\bf Problem 6.} If $D$ has asymptotic angle $\beta <\pi$, and $u$ is a solution to (\ref{eq:bdryvalueprob}), then
must it be that $u\equiv 0$?

\bigskip

\bibliographystyle{amsplain}

\end{document}